\documentclass{article}
\usepackage{color}
\usepackage{graphicx}
\usepackage{amsmath}
\usepackage{amssymb}
\usepackage{mathrsfs}
\usepackage[numbers,sort&compress]{natbib}
\usepackage{amsthm}
\usepackage{pgf}
\usepackage{graphicx}
\usepackage{tikz}

\usepackage{graphicx}
\usepackage{hyperref}
\usepackage[T1]{fontenc}
\usepackage[english]{babel}
\usepackage{listings}
\usepackage{xcolor,mathrsfs,url}
\usepackage{amssymb}
\usepackage{amsmath}
\usepackage{ifthen}
\usepackage{caption}
\usepackage{float}

{\LARGE }

\begin{document}
	\title{Long-time asymptotic behavior of the nonlocal nonlinear Schr\"odinger equation with initial potential in weighted sobolev space}
	\author{Meisen Chen,   En-Gui Fan\thanks{Corresponding author: faneg@fudan.edu.cn}\\[4pt]
		School of Mathematical Sciences, Fudan University,\\
		Shanghai 200433, P.R. China}
	\maketitle
	
	\theoremstyle{plain}
	\newtheorem{proposition}{Proposition}[section]
	\newtheorem{lemma}[proposition]{Lemma}
	\newtheorem{drhp}[proposition]{$\bar\partial$-RH problem}
	\newtheorem{rhp}[proposition]{RH problem}
	\newtheorem{dbarproblem}[proposition]{$\bar\partial$-problem}
	\newtheorem{remark}[proposition]{Remark}

\begin{abstract}
	In this paper, we are going to investigate Cauchy problem for nonlocal nonlinear Schr\"odinger equation with the initial potential $q_0(x)$ in weighted sobolev space $H^{1,1}(\mathbb{R})$,
	\begin{align}
		iq_t(x,t)&+q_{xx}(x,t)+2\sigma q^2(x,t)\bar q(-x,t)=0,\quad\sigma=\pm1,\nonumber \\
		q(x,0)&=q_0(x).\nonumber
	\end{align}
	 We show that the solution can be represented by the solution of a Riemann-Hilbert problem (RH problem), and assuming no discrete spectrum, we majorly apply $\bar\partial$-steepest cescent descent method on analyzing the long-time asymptotic behavior of it. 
\end{abstract}

\baselineskip=18pt

\newpage
\tableofcontents

\section{Introduction}
\indent

The nonlocal nonlinear Schr\"odinger (NNLS) equation
	\begin{align}
		iq_t(x,t)&+q_{xx}(x,t)+2\sigma q^2(x,t)\bar q(-x,t)=0,\quad\sigma=\pm1,\label{e1}
	\end{align}
was first introduced by Ablowitz and Musslimani in 2013 \cite{Ablowitz2013integrable} .
It is an integrable system with Lax pair, and its inverse scattering transformations with zero and nonzero boundary condition have been completed by Ablowitz et al \cite{Ablowitz2016inverse, Ablowitz2018inverse}.
The long-time asymptotic analysis for the NNLS equation with rapidly decaying initial data was given by Rybalko and Shepelsky \cite{Rybalko2019long}.
In 2018, Feng et al have got the general soliton for the NNLS equation by Hirota's bilinear method and the
Kadomtsev-Petviashvili hierarchy reduction method \cite{Feng2018general}.   We note that the NNLS
have the $\mathcal{PT}$-symmetry \cite{bender1998real} potential $V(x,t)=q(x,t)\bar q(-x,t)$: $V(x,t)=\overline{V}(-x,t)$.
Recent years, there are many works for the $\mathcal{PT}$-symmetric equations \cite{Gerdjikov2016the, Ablowitz2017integrable,Sinha2017integrable,Song2017solitons,Ablowitz2014integrable}.
$\mathcal{PT}$ symmetry is also an important conception in optics \cite{Ruter2010observation,Regensburger2012parity,Regensburger2013observation}.

In this paper, we  apply   a systematic dbar-steepest cescent method to analyze the long-time asymptotic behavior of the solution for the NNLS equation (\ref{e1})
with the initial potential
	\begin{align}
		q(x,0)&=q_0(x) \in H^{1,1}(\mathbb{R}),  \label{e2}
	\end{align}
where   $H^{1,1}(\mathbb{R})$ is a  weighted Sobolev space  defined by
\begin{align*}
	&H^{1,1}(\mathbb{R})=L^{2,1}(\mathbb{R})\cap H^1(\mathbb{R}), \  \
	L^{2,1}(\mathbb{R}) =\{(1+|\cdot|^2)^\frac{1}{2}f\in L^2(\mathbb{R})\},\\
	&H^1(\mathbb{R})=\left\{f\in L^2(\mathbb{R})|f'\in L^2(\mathbb{R})\right\}.
\end{align*}
For   $f\in L^{2,1}(\mathbb{R})$,   its  norm is defined by  $\parallel f\parallel_{2,1}=\parallel(1+|\cdot|^2)^{1/2}f\parallel_2$.

The dbar-steepest descent method, developed from the Deift-Zhou steepest descent method,
is  very powerful in analyzing the long-time asymptotic behavior with potential in weighted Sobolev space \cite{deift1993long,deift1993steepest,Borghese2018long,Liu2018long,Giavedoni2017long}.
Cuccagna et al also use it to analyze the asymptotic  stability for  the soliton solutions of the NLS equation   \cite{Cuccagna2014the,Cuccagna2016on}.

As shown in \cite{zhou1998l2}, there is a bijection   map   between space of initial potential and space of reflection coefficients:
\begin{align}\label{e3}
	H^{1,1}(\mathbb{R})\to H^{1,1}(\mathbb{R}):\quad q(x)\mapsto \{ r(z),\breve{r}(z)\},
\end{align}
then, we can set the initial data in $H^{1,1}(\mathbb{R})$.
Indeed, in this paper, the proof in this article only require $\{r(z),\breve r(z)\}\subset H^1(\mathbb{R})$; however, by simply calculation, $\{r(z),\breve r(z)\}$ do not belong to $H^1(\mathbb{R})$ with time
evolution but persist in $H^{1,1}(\mathbb{R})$ as time evolving, seeing Remark \ref{r2.2};
so, by (\ref{e3}), we restrict our initial potential $q_0(x)\in H^{1,1}(\mathbb{R})$.
We also restrict the potential $q_0(x)$ to be generic: for the direct scattering, $q(x)\mapsto \{r(z),\breve r(z)\}$, the reflection coefficient do not possess any singular point along the continuous spectrum.
Moreover, we also assume that the scattering coefficient $a(z)$ possesses no zero on $\mathbb{C}_+$ while $\breve{a}(z)$ has not any zero on $\mathbb{C}_-$.

This article is organized as follows.
At section \ref{S2}, we simply display the main result of the direct scattering.
At section \ref{S3}, we construct the Riemann-Hilbert (RH)  problem based on the   Lax pair ({\ref{eq:3}}).
At section \ref{Aa}, we establish the map $q(x)\mapsto \{ r(z),\breve r(z) \}$.
At section \ref{s4}, we  carry out a series of RH problem transformations:  $M\rightsquigarrow M^{(1)}\rightsquigarrow M^{(2)}$  and factorizing $M^{(2)}$ into a product of one model RH problem $ M^{(2)}_\text{\tiny{RHP}}$ and one pure $\bar\partial$ problem $M^{(3)}$.
At section \ref{S5}, we get the long-time asymptotic behavior of  the NNLS equation.

\section{Direct Scattering Problem}\label{S2}
\indent

In this section, we state the main result of the direct scattering transformation. Based on the Lax pair (\ref{eq:3}), we obtain the Jost solution, modified Jost solution,
 the scattering matrix, their symmetric properties and their asymptotic properties as $z\to\infty$.

The NNLS equation admits the Lax pair:
\begin{subequations}\label{eq:3}
	\begin{align}
	&\phi_x+iz\sigma_3\phi=Q\phi,&&  Q\equiv Q(x,t)=\left(\begin{matrix}
	0&q(x,t)\\-\sigma \bar q(-x,t)&0
	\end{matrix}\right),\\
	&\phi_t+2iz^2\sigma_3\phi=P\phi,&&P=i\sigma_3(Q_x-Q^2)+2zQ,
	\end{align}
\end{subequations}
where $\sigma_3=\left(\begin{matrix}1&0\\0&-1
\end{matrix}\right)$ is a Pauli matrix, $\{x,t\}\subset\mathbb{R}$ and $z\in\mathbb{C}$ denotes the spectrum.
The Lax pair (\ref{eq:3}) admits the Jost solution $\phi^\pm\equiv\phi^\pm(z;x,t)$:
\begin{align}\label{e5r}
	&\phi^\pm\sim e^{-t\varphi\sigma_3},\quad\text{as}\quad x\to\pm\infty,
\end{align}
where $\varphi\equiv\varphi(z;x,t)=i(zx/t+2z^2 )$ is the phase function.
Then, it's natural for us to introduce the modified Jost solution $\Phi^\pm\equiv\Phi^\pm(z;x,t)$:
\begin{align}\label{e:4}
\Phi^\pm=\phi^\pm e^{t\varphi\sigma_3},
\end{align}
 such that
 \begin{align}
\Phi^\pm\sim I,\quad\text{as}\quad x\to\pm\infty,
\end{align}
We can derive that $\Phi^\pm(z,x)=\Phi^\pm(z;x,t)$ satisfy the following Volterra integral equations associated with (\ref{eq:3}):
\begin{align}\label{e5}
	\Phi^\pm(z,x)=I+\int_{\pm\infty}^{x}e^{iz(y-x)\hat\sigma_3}[Q(y)\Phi^\pm(z,y)]\mathrm{d}y,\quad z\in\mathbb{C}.
\end{align}
Because $q(x)\in L^{2,1}(\mathbb{R})$, and by Schwartz inequality,
\begin{align*}
	\parallel q\parallel_1\le\parallel q\parallel_{2,1}\left(\int_\mathbb{R}(1+x^2)^{-1}\right)^\frac{1}{2}\le\sqrt{\pi}\parallel q\parallel_{2,1},
\end{align*}
the $L^1$-norm of $q(x)$ is bounded.
By taking Neumann series of $\Phi^\pm $ in the Volterra integral, we naturally obtain analytic properties of $\Phi^\pm$ that are given  in the following Proposition \ref{p2.1}. We write $\Phi^\pm=(\Phi^\pm_1,\Phi^\pm_2)$.
\begin{proposition}\label{p2.1}
	For the potential $q(x)\in L^{2,1}(\mathbb{R})$, $(\Phi^-_1(z,x),\Phi^+_2(z,x))$ is uniquely defined and analytic on $\mathbb{C}_+=\{ {\rm Im} z>0\}$, and continuously extended to $z\in\mathbb{C}_+ \cup \mathbb{R}$;
In the meanwhile, $(\Phi^+_1(z,x),\Phi^-_2(z,x))$ is uniquely defined and analytic on $\mathbb{C}_-=\{{\rm Im}  z<0\}$, and continuously extended to $z\in\mathbb{C}_- \cup \mathbb{R}$. See $\mathbb{C}_+$ and $\mathbb{C}_-$ at Figure \ref{fig1}.
\end{proposition}
\begin{figure}[h]
	\centering{\includegraphics[width=0.7\linewidth]{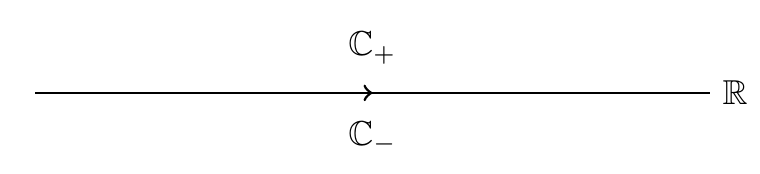}}
	\caption{\label{fig1}$\mathbb{C}_+=\{ {\rm Im} z>0\}$, $\mathbb{C}_-=\{ {\rm Im} z<0\}$ and the real line $\mathbb{R}$ considered as the continuous spectrum}
\end{figure}

Noticing in Lax pair  (\ref{eq:3}) that
\begin{align*}
	\text{tr}(Q-iz\sigma_3)=0,
\end{align*}
since $\phi^+$ and $\phi^-$ both solve  the Lax pair  (\ref{eq:3}),  we have that $\det\phi^\pm\equiv1$ and there exists a unique $2\times2$ matrix $S\equiv S(z)$ independent on $(x,t)$ such that
\begin{align}\label{e6}
	\phi^-(z;x,t)=\phi^+(z;x,t)S(z),\quad S(z)=\left(\begin{matrix}
		a(z)&\breve{b}(z)\\b(z)&\breve{a}(z)
	\end{matrix}\right),
\end{align}
where $S(z)$ is well known as scattering matrix and $a(z)$, $\breve{a}(z)$, $b(z)$, $\breve{b}(z)$ are so called scattering coefficients. By basic linear algebra, we derive from (\ref{e:4}) and (\ref{e6})
\begin{subequations}\label{e7}
\begin{align}
	a(z)&=\det(\Phi^-_1,\Phi^+_2),\ \ b(z)=\det(\Phi^+_1,\Phi^-_1)e^{-2t\varphi},\label{e7a}\\
	\breve{a}(z)&=\det(\Phi^+_1,\Phi^-_2),\ \ \breve{b}(z)=\det(\Phi^-_2,\Phi^+_2)e^{2t\varphi},\label{e7b}
\end{align}
\end{subequations}
which implies, by Proposition \ref{p2.1}, that $a(z)$ and $\breve{a}(z)$  are  analytic on $\mathbb{C}_+$ and $\mathbb{C}_-$, respectively,
and continuously extended to $\mathbb{C}_+\cup\mathbb{R}$ and $\mathbb{C}_-\cup\mathbb{R}$ respectively;
in the meanwhile,  both $b(z)$ and $\breve{b}(z)$ are continuous on $\mathbb{R}$.

By WKB expansion, someone can get the asymptotic property of the modified Jost solutions:
\begin{align}\label{e9r}
	\Phi^\pm(z;x,t)\sim I+\mathcal{O}(z^{-1}),\quad\text{as}\quad z\to\infty,
\end{align}
which with (\ref{e7}) implies that as $z\to\infty$,
\begin{subequations}\label{e9}
\begin{align}
	a(z)&\sim 1+\mathcal{O}(z^{-1}),\quad b(z)\sim \mathcal{O}(z^{-1}),\label{e9a}\\
	\breve{a}(z)&\sim 1+\mathcal{O}(z^{-1}),\quad\breve{b}(z)\sim \mathcal{O}(z^{-1}),
\end{align}
\end{subequations}
in addition, we can write the asymptotic property of $\Phi^\pm_{1,2}$ more precisely,
\begin{align}\label{e11}
	\Psi_{1,2}^\pm(z;x,t)\sim-\frac{ i}{2}q(x,t)z^{-1}+\mathcal{O}(z^{-2}),\quad\text{as}\quad z\to\infty.
\end{align}

Defining the reflection coefficients,
\begin{align}\label{e12r}
	r\equiv r(z)=\frac{b(z)}{a(z)}, \quad\breve r\equiv\breve{r}(z)=\frac{\breve{b}(z)}{\breve{a}(z)}, \quad\text{on}\quad z\in\mathbb{R},
\end{align}
we can see that if $q_0(x)\in L^{2,1}(\mathbb{R})$, these reflection coefficients belong to $H^1(\mathbb{R})$. See more detail in Section \ref{Aa}. Then, both $r(z)$ and $\bar r(z)$ are $\frac{1}{2}$-H\"older continuous on the real line and satisfy
\begin{align}\label{e11s}
	|r(z)|, |\breve{r}(z)|\lesssim (1+z^2)^{-\frac{1}{4}},\quad z\in\mathbb{R},
\end{align}
therefore, both $r(z)$ and $\breve{r}(z)$ are continuous and bounded on $z\in\mathbb{R}$; moreover, recalling (\ref{e6}) and that $\det\phi^\pm\equiv 1$, we have
\begin{align*}
	\det S(z)=a(z)\breve a(z)-b(z)\breve b(z)=1,
\end{align*}
then,
\begin{align*}
	1-r(z)\breve{r}(z)=\frac{1}{a(z)\breve{a}(z)},
\end{align*}
and, with the generic assumption, we obtain that $1-r(z)\breve{r}(z)$ is also bounded, continuous and non-vanishing on the real line.

By basic calculation, $\phi^\pm$ admits symmetry:
\begin{align}\label{e16r}
	\Lambda\overline{\phi^\pm(-\bar z;-x,t)}\Lambda^{-1}=\phi^\mp(z;x,t),\quad \Lambda=\left(\begin{matrix}
		0&\sigma\\1&0
	\end{matrix}\right),\quad z\in\mathbb{C},
\end{align}
which implies the symmetry for the corespondent scattering coefficients:
	\begin{align}\label{e13b}
		&a(z)=\overline{a(-z)},\quad \breve{a}( z)=-\overline{\breve{a}(-z)},\quad b(z)=-\sigma\overline{\breve{b}(-z)},\quad z\in\mathbb{R}.
	\end{align}

\begin{remark}\label{r2.2}
	By (\ref{e:4}), (\ref{e7a}) and (\ref{e12r}), sure we have
	\begin{align}\label{e15}
		r&=\frac{\det(\phi^+_1,\phi^-_1)}{\det(\phi^-_1,\phi^+_2)}.
	\end{align}
If we set $q\big|_{t=t_0}$ as the initial data for $t_0>0$ and $\tilde\phi^\pm\equiv\tilde\phi^\pm(z;x,t)$ as the corespondent Jost solution such that
\begin{align}\label{e16}
	\tilde\phi^\pm\sim e^{-i(zx+2z^2(t-t_0))\sigma_3},\quad x\to\pm\infty,
\end{align}
then, there is a corespondent reflection coefficient $\tilde r\equiv\tilde r(z)$ satisfying:
\begin{align}\label{e17s}
	\tilde r=\frac{\det(\tilde\phi^+_1,\tilde\phi^-_1)}{\det(\tilde\phi^-_1,\tilde\phi^+_2)}.
\end{align}
Comparing (\ref{e5r}), (\ref{e15}) with (\ref{e16}), (\ref{e17s}) respectively, we obtain the relation of $r(z)$ and $\tilde r(z)$ by uniqueness of the Jost solution:
\begin{align*}
	\tilde r(z)=r(z)e^{4t_0z^2},
\end{align*}
which means that $r(z)$ persists in $H^{1,1}(\mathbb{R})$ by simple computation. Of course, $\breve r(z)$ also does possess this property.
\end{remark}

\section{RH problem}\label{S3}
\indent

In this section, we construct the corespondent RH problem for the Lax pair and the reconstructed formula (\ref{e17}) for $q(x)$.

Introducing
\begin{align}\label{e12}
	M\equiv M(z;x,t)=\begin{cases}
		\left(\frac{\Phi_1^-(z;x,t)}{a(z)},\Phi_2^+(z;x,t)\right),\quad z\in\mathbb{C}_+,\\
		\left(\Phi_1^+(z;x,t),\frac{\Phi_2^-(z;x,t)}{\breve{a}(z)}\right), \quad z\in\mathbb{C}_-,
	\end{cases}
\end{align}
and the jump contour $\mathbb{R}$,
we observe from Section \ref{S2} that $M$ admits the following RH problem.
\begin{rhp}\label{rh3.1}
	Find a $2\times2$ matrix function on $\mathbb{C}\setminus\mathbb{R}$ such that:
\begin{itemize}
	\item Analyticity: $M$ is holomorphic on $\mathbb{C}\setminus \mathbb{R}$.
	\item Normalization:
	\begin{align*}
		M\sim I+\mathcal{O}(z^{-1}), \quad\text{as}\quad z\to\infty.
	\end{align*}
    \item Jump condition:
    \begin{align*}
    	M_+=M_-V,\quad\text{on}\quad \mathbb{R},
    \end{align*}
where
\begin{align*}
	V=e^{t\varphi\hat\sigma_3}\left(\begin{matrix}
		1-r\breve{r}&-\breve{r}\\r&1
	\end{matrix}\right).
\end{align*}
\end{itemize}
\end{rhp}

In addition, we get the reconstructed formula for the potential by (\ref{e9a}), (\ref{e11}) and (\ref{e12}),
\begin{align}\label{e17}
	q(x,t)=2i\lim_{z\to\infty}zM_{1,2}(z;x,t).
\end{align}

\section{Analysis on scattering maps: $q(x)\mapsto \{ r(z), \breve{r}(z) \}$ } \label{Aa}
\indent

In this section, we focus on the map from initial data to reflection coefficients.
Here, we denote $q_0(x)$ by $q(x)$  without confusion of notation.
\begin{proposition}\label{pa1}
	If $q(x)\in L^{2,1}(\mathbb{R})$, then the reflection coefficient $r(z)$ and $\breve{r}(z)$ both belong to $H^1(\mathbb{R})$.
\end{proposition}

\subsection{ $r(z)\in H^1(\mathbb{R})$}
\indent

By (\ref{e7}), (\ref{e9}) and Proposition \ref{p2.1}, we learn that $a(z)$, $b(z)$, $\breve a(z)$ and $\breve b(z)$ are continuous and bounded on $z\in\mathbb{R}$.
Since the assumption that $q(x)$ is generic and the fact that
\begin{align*}
	r(z)=\frac{b(z)}{a(z)}, \quad r'(z)=\frac{1}{a(z)^2}(b'(z)a(z)-a'(z)b(z)),
\end{align*}
by the boundedness of $a(z)$ and $b(z)$, we learn that $r(z)\in H^1(\mathbb{R})$ only if
\begin{align}
	a'(z),b(z),b'(z)\in L^2(\mathbb{R}).\label{e52}
\end{align}

Introducing
\begin{align}\label{e24}
	Y^\pm\equiv Y^\pm(z,x)=e^{ixz\hat\sigma_3}\Phi^\pm(z;x,0)=e^{ixz\sigma_3}\phi^\pm(z;x,0),
\end{align}
by Proposition and (\ref{e9r}), $Y^\pm$ is bounded on $z\in\mathbb{R}$. By (\ref{e5}) and (\ref{e24}), we get the Volterra integral for $Y^\pm$:
\begin{align}\label{e53}
	Y^\pm(z,x)&=I+\int_{\pm\infty}^{x}e^{iyz\hat\sigma_3}Q(y)Y^\pm(z,y)\mathrm{d}y,
\end{align}
and by (\ref{e7a}), (\ref{e16r}) and (\ref{e24}), we have
\begin{subequations}\label{e54}
	\begin{align}
		a(z)=&Y_{1,1}^-(z,x)Y_{1,1}^-(-z,-x)+Y_{2,1}^-(z,x)Y_{2,1}^-(-z,-x),\label{e54a}\\
		a'(z)=&\partial_zY_{1,1}^-(z,x)Y_{1,1}^-(-z,-x)+\partial_zY_{2,1}^-(z,x)Y_{2,1}^-(-z,-x)\notag\\
		&-Y_{1,1}^-(z,x)\partial_zY_{1,1}^-(-z,-x)-Y_{2,1}^-(z,x)\partial_zY_{2,1}^-(-z,-x),\label{e54b}\\
		b(z)=&Y_{1,1}^+(z,x)Y_{2,1}^-(z,x)-Y_{1,1}^-(z,x)Y_{2,1}^+(z,x),\\
		b'(z)=&\partial_zY_{1,1}^+(z,x)Y_{2,1}^-(z,x)-\partial_zY_{1,1}^-(z,x)Y_{2,1}^+(z,x)\notag\\
		&+Y_{1,1}^+(z,x)\partial_zY_{2,1}^-(z,x)-Y_{1,1}^-(z,x)\partial_zY_{2,1}^+(z,x).
	\end{align}
\end{subequations}
Therefore, seeing from (\ref{e54}), (\ref{e52}) is the consequence of the boundedness of $Y^\pm$ on $z\in\mathbb{R}$ and Lemma \ref{la2}.
\begin{lemma}\label{la2}
	If $q(x)\in L^{2,1}(\mathbb{R})$, then $\left\{Y_1^\pm-\left(\begin{matrix}
		1\\0
	\end{matrix}\right),\partial_z Y^\pm_{1}\right\}\subset\mathcal{A}=L^\infty(\mathbb{R},L^2(\mathbb{R}))$, where for $f(z,x)=\left(\begin{matrix}
	f_1(z,x)\\f_2(z,x)
\end{matrix}\right)\in\mathcal{A}$,
	\begin{align*}
		\parallel f\parallel_\mathcal{A}=\sup_{x\in\mathbb{R}}\parallel f(\cdot,x)\parallel_2.
	\end{align*}
In this article, without confusion of notation, we will denote $L^2(\mathbb{R})$ and $(L^2(\mathbb{R}))^2$ uniformly by $L^2(\mathbb{R})$, and the norm of $f(\cdot,x)\in(L^2(\mathbb{R}))^2$ is denoted by
\begin{align*}
	\parallel f(\cdot,x)\parallel_2=(\parallel f_1(\cdot,x)\parallel_2^2+\parallel f_1(\cdot,x)\parallel_2^2)^\frac{1}{2}.
\end{align*}
\end{lemma}
Before the proof of Lemma \ref{la2}, we introduce integral operators $T^\pm$ such that:
\begin{align*}
	(T^\pm f)(z,x)=\int_{\pm\infty}^{x}e^{iyz\sigma_3}Q(y)f(z,y)\mathrm{d}y.
\end{align*}
Then, it can be derived from (\ref{e53}) that
\begin{subequations}\label{e56}
	\begin{align}
		&Y^\pm_1(z,x)-\left(\begin{matrix}
			1\\0
		\end{matrix}\right)=\sum_{n=1}^\infty \left[(T^\pm)^n\left(\begin{matrix}
			1\\0
		\end{matrix}\right)\right](z,x),\\
		&\partial_zY^\pm_1(z,x)=\sum_{n=1}^\infty \left[(T^\pm)^n\left(\begin{matrix}
			1\\0
		\end{matrix}\right)\right]'_z(z,x).
	\end{align}
\end{subequations}
To obtain the result in Lemma \ref{la2}, we estimate $\mathcal{A}$-norm for each element of the summation appearing in the right hand of (\ref{e56}), which is shown in Proposition \ref{pa3s} and \ref{pa3}.
\begin{proposition}\label{pa3s}
	If $q(x)\in L^{2,1}(\mathbb{R})$, some ones obtain estimates for the $\mathcal{A}$-norm of $(T^\pm)^n\left(\begin{matrix}
		1\\0
	\end{matrix}\right)$:
	\begin{align*}
		&\Big\|(T^\pm)^{2n-1}\left(\begin{matrix}
			1\\0
		\end{matrix}\right)\Big\|_\mathcal{A}\le \sqrt{\pi}\parallel q\parallel_{2}\frac{\parallel q\parallel_{1}^{2n-2}}{(n-1)!},\\
		&\Big\|(T^\pm)^{2n}\left(\begin{matrix}
			1\\0
		\end{matrix}\right)\Big\|_\mathcal{A}\le \sqrt{\pi}\parallel q\parallel_{2}\frac{\parallel q\parallel_{1}^{2n-1}}{(n-1)!},
	\end{align*}
	where $n=2,\ 3,\dots$.
\end{proposition}
\begin{proposition}\label{pa3}
	If $q(x)\in L^{2,1}(\mathbb{R})$, some ones obtain estimates for the $\mathcal{A}$-norm of $\left[(T^\pm)^n\left(\begin{matrix}
		1\\0
	\end{matrix}\right)\right]'_z$:
	\begin{align*}
		&\Big\|\left[T^\pm\left(\begin{matrix}
			1\\0
		\end{matrix}\right)\right]'_z\Big\|_\mathcal{A}\le C\parallel q\parallel_{2,1},\\
		&\Big\|\left[(T^\pm)^2\left(\begin{matrix}
			1\\0
		\end{matrix}\right)\right]'_z\Big\|_\mathcal{A}\le C\parallel q\parallel_{2,1}\parallel q\parallel_1,\\
		&\Big\|\left[(T^\pm)^{2n-1}\left(\begin{matrix}
			1\\0
		\end{matrix}\right)\right]'_z\Big\|_\mathcal{A}\le C\parallel q\parallel_{2,1}\frac{\parallel q\parallel_{1}^{2n-2}}{(n-2)!},\\
		&\Big\|\left[(T^\pm)^{2n}\left(\begin{matrix}
			1\\0
		\end{matrix}\right)\right]'_z\Big\|_\mathcal{A}\le C\parallel q\parallel_{2,1}\frac{\parallel q\parallel_{1}^{2n-1}}{(n-2)!},
	\end{align*}
	where $n=2,3,\dots$ and $C$ is some fixed positive number.
\end{proposition}
\begin{proof}[proof of Proposition \ref{pa3s}]
	In functional analysis, there is an important fact that for any function $g\in L^2(\mathbb{R})$, the $L^2$-norm of $g$ can be written as
	\begin{align}\label{e57s}
		\parallel g\parallel_2=\sup_{h\in L^2}\int_\mathbb{R}g(s)\overline{h(s)}\mathrm{d}s,
	\end{align}
	which is very useful in our proof. Without loss of generality, we only check the result for $T^-$ in detail, and remaining results is similarly obtained. By definition of $T^-$, we have
	\begin{subequations}\label{e57}
		\begin{align}
			&(T^-)^{2n-1}\left(\begin{matrix}
				1\\0
			\end{matrix}\right)(z,x)=\int_{-\infty}^{x}\int_{-\infty}^{y_1}\dots\int_{-\infty}^{y_{2n-2}}\left(\begin{matrix}
				0\\K_{2n-1}
			\end{matrix}\right)\mathrm{d}y_{2n-1}\dots\mathrm{d}y_1,\label{e57a}\\
			&\quad K_{2n-1}=\prod_{k=1}^{n}(-\bar q(-y_{2k-1}))\prod_{k=1}^{n-1}q(y_{2k})e^{2izA_{2n-1}},\notag\\
			&(T^-)^{2n}\left(\begin{matrix}
				1\\0
			\end{matrix}\right)(z,x)=\int_{-\infty}^{x}\int_{-\infty}^{y_1}\dots\int_{-\infty}^{y_{2n-1}}\left(\begin{matrix}
				K_{2n}\\0
			\end{matrix}\right)\mathrm{d}y_{2n}\dots\mathrm{d}y_1,\\
			&\quad K_{2n}=\prod_{k=1}^{n}(-\bar q(-y_{2k-1}))\prod_{k=1}^{n}q(y_{2k})e^{-2izA_{2n}},\notag
		\end{align}
	\end{subequations}
	where
	\begin{align*}
		A_{n}=\sum_{k=1}^{n}(-1)^ky_k,\quad
		n=1,2,\dots.
	\end{align*}
	By (\ref{e57s}), Schwartz inequality and Fourier transformation, we have
	\begin{align*}
		\Big\| (T^-)^{2n-1}\left(\begin{matrix}
			1\\0
		\end{matrix}\right)\Big\|_\mathcal{A}=&\sqrt{\pi}\sup_{x\in\mathbb{R},\parallel f\parallel_2\le1}\int_{-\infty}^{x}\int_{-\infty}^{y_1}\dots\int_{-\infty}^{y_{2n-2}}\prod_{k=1}^{n-1}\left(\bar q(-y_{2k-1})q(y_{2k})\right)\times\\
		&(-1)^n(\bar q(-y_{2n-1}))\hat f\left(-A_{2n-1}\right)\mathrm{d}y_{2n-1}\dots\mathrm{d}y_1\\
		\le&\sqrt{\pi} \sup_{x\in\mathbb{R},\parallel f\parallel_2\le1}\int_{-\infty}^{x}\dots\int_{-\infty}^{y_{2n-3}}\prod_{k=1}^{n-1}|\bar q(-y_{2k-2})q(y_{2k})|\times\\
		&\parallel q\parallel_{2}\parallel \hat f\parallel_2\mathrm{d}y_{2n-2}\dots\mathrm{d}y_1\\
		\le&\sqrt{\pi}\parallel q\parallel_{2}\sup_{x\in\mathbb{R}}\int_{-\infty}^{x}\dots\int_{-\infty}^{y_{n-2}}\prod_{k=1}^{n-1}|q(-y_{k})|\mathrm{d}y_{n-1}\dots\mathrm{d}y_1\times\\
		&\int_{-\infty}^{x}\dots\int_{-\infty}^{y_{n-2}}\prod_{k=1}^{n-1}|q(y_{k})|\mathrm{d}y_{n-1}\dots\mathrm{d}y_1\\
		\le&\sqrt{\pi}\parallel q\parallel_{2}\frac{\parallel q\parallel_1^{2n-2}}{(n-1)!},
	\end{align*}
	where $\hat f$ is the Fourier transformation of $f$:
	\begin{align*}
		\hat f(\zeta)=\pi^{-\frac{1}{2}}\int_{\mathbb{R}}f(z)e^{-2iz\zeta}\mathrm{d}z.
	\end{align*}
	The result for $(T^-)^{2n}\left(\begin{matrix}
		1\\0
	\end{matrix}\right)$ can be obtained similarly, and then the proposition is confirmed.
\end{proof}

\begin{proof}[proof of Proposition \ref{pa3}]
	Without loss of generality, we also only give the proof of this proposition for $T^-$.
    Taking the derivative in (\ref{e57a}), it follows that
		\begin{align}
			&\left[(T^-)^{2n-1}\left(\begin{matrix}
				1\\0
			\end{matrix}\right)\right]'_z(z,x)=\int_{-\infty}^{x}\int_{-\infty}^{y_1}\dots\int_{-\infty}^{y_{2n-2}}\left(\begin{matrix}
				0\\K_{2n-1}'
			\end{matrix}\right)\mathrm{d}y_{2n-1}\dots\mathrm{d}y_1,\\
			&\quad K_{2n-1}'=2iA_{2n-1}\left(\prod_{k=1}^{n}(-\bar q(-y_{2k-1}))\right)\left(\prod_{k=1}^{n-1}q(y_{2k})\right)e^{2izA_{2n-1}},\notag
		\end{align}
	where $n=1,2,\dots$.
	We split $\left[(T^-)^{2n-1}\left(\begin{matrix}
		1\\0
	\end{matrix}\right)\right]'_z$ into $2n-1$ terms:
	\begin{subequations}
		\begin{align}
			&\left[(T^-)^{2n-1}\left(\begin{matrix}
				1\\0
			\end{matrix}\right)\right]'_z=\sum_{j=1}^{2n-1}\left(\begin{matrix}
				0\\g_j
			\end{matrix}\right),\label{e60a}\\
			&g_j(z,x)=\int_{-\infty}^{x}\int_{-\infty}^{y_1}\dots\int_{-\infty}^{y_{2n-2}}2i(-1)^jy_j\left(\prod_{k=1}^{n}(-\bar q(-y_{2k-1}))\right)\times\notag\\
			&\quad\quad\left(\prod_{k=1}^{n-1}q(y_{2k})\right)e^{2izA_{2n-1}}\mathrm{d}y_{2n-1}\dots\mathrm{d}y_1,\quad j=1,\dots,2n-1,
		\end{align}
	\end{subequations}
	and $\parallel g_j\parallel_\mathcal{A}$ is bounded by $\parallel q\parallel_{2,1}$ and $\parallel q\parallel_1$:
	\begin{align*}
		\parallel g_j\parallel_\mathcal{A}\le \sqrt{\pi}\parallel q\parallel_{2,1}\frac{\parallel q\parallel_1^{2n-2}}{(n-1)!},
	\end{align*}
	which can be verified by strictly applying the technique to bound $\Big\| (T^-)^{2n-1}\left(\begin{matrix}
		1\\0
	\end{matrix}\right)\Big\|_\mathcal{A}$ in Proposition \ref{pa3s}. Then, the result for $\Big\|\left[(T^-)^{2n-1}\left(\begin{matrix}
	1\\0
\end{matrix}\right)\right]'_z\Big\|_\mathcal{A}$ immediately follows. The result for $\Big\|\left[(T^-)^{2n}\left(\begin{matrix}
1\\0
\end{matrix}\right)\right]'_z\Big\|_\mathcal{A}$ is obtained similarly.
\end{proof}
According to control convergence theorem and the fact that
\begin{align*}
	\parallel q\parallel_1\le \parallel q\parallel_{2,1}\left(\int_{\mathbb{R}}(1+x^2)^{-1}\mathrm{d}x\right)^\frac{1}{2}\le\sqrt{\pi}\parallel q\parallel_{2,1},
\end{align*}
we conclude from Proposition \ref{pa3s} and \ref{pa3} that $Y^\pm_1-\left(\begin{matrix}
	1\\0
\end{matrix}\right)$ and $\partial_zY^\pm_1$ exist in $\mathcal{A}$, and the $\mathcal{A}$-norm of them satisfy
\begin{align*}
	&\Big\| Y^\pm_1-\left(\begin{matrix}
		1\\0
	\end{matrix}\right)\Big\|_\mathcal{A}\le\sqrt{\pi}\parallel q\parallel_2(1+\parallel q\parallel_1)e^{\parallel q\parallel_1^2}
    =\sqrt{\pi}\parallel q\parallel_2(1+\sqrt{\pi}\parallel q\parallel_{2,1})e^{\pi\parallel q\parallel_{2,1}^2},\\
	&\parallel\partial_zY^\pm_1\parallel_\mathcal{A}\le C\parallel q\parallel_{2,1}\parallel q\parallel_1e^{\parallel q\parallel_1^2}(2+\parallel q\parallel_1+\parallel q\parallel_1^2)\\
	&\quad\quad\quad\quad\quad\lesssim\parallel q\parallel_{2,1}^2(2+\sqrt{\pi}\parallel q\parallel_{2,1}+\pi\parallel q\parallel_{2,1}^2)e^{\pi\parallel q\parallel_{2,1}^2}.
\end{align*}
Then, the result of Lemma \ref{la2} is verified.

\subsection{ $\breve{r}(z)\in H^1(\mathbb{R})$}
\indent

Because of the assumption that $q(x)$ is generic and the fact that for $z\in\mathbb{R}$, by (\ref{e13b}),
\begin{align*}
	\breve{r}(z)=\frac{\breve{b}(z)}{\breve{a}(z)}=-\sigma\frac{\overline{b(-z)}}{\breve{a}(z)},
\end{align*}
$\breve{r}(z)\in L^2(\mathbb{R})$ follows immediately after that $b(z)\in L^2(\mathbb{R})$. Also, for $\breve{r}'(z)$, we have
\begin{align*}
	\breve{r}'(z)=\sigma\frac{\overline{b'(-z)}}{\breve{a}(z)}+\frac{\breve{a}'(z)\breve{b}(z)}{\breve{a}^2(z)},
\end{align*}
then, because of that $\breve a(z)$ and $\breve b(z)$ are bounded on $\mathbb{R}$, $b'(z)\in L^2(\mathbb{R})$ and that $q(x)$ is generic, $\breve{r}'(z)\in L^2(\mathbb{R})$ only if $\breve{a}'(z)\in L^2(\mathbb{R})$. By (\ref{e7b}), (\ref{e16r}) and (\ref{e24}), we have
\begin{align}\label{e62}
	\breve{a}'(z)=&\partial_zY_{1,1}^+(z,x)Y_{1,1}^+(-z,-x)+\partial_zY_{2,1}^+(z,x)Y_{2,1}^+(-z,-x)\notag\\
	&-Y_{1,1}^+(z,x)\partial_zY_{1,1}^+(-z,-x)-Y_{2,1}^+(z,x)\partial_zY_{2,1}^+(-z,-x).
\end{align}
Therefore, $\breve{a}'(z)\in L^2(\mathbb{R})$ is the consequence of (\ref{e62}), the boundedness of $Y^\pm$ and Lemma \ref{la2}. To sum up, we have completed the proof of $\breve{r}(z)\in H^1(\mathbb{R})$.

\section{Deformations for RH problem}\label{s4}
\indent

In this section, we deform $M$ several time such that the final RH problem satisfies a model RH problem. 
From $\varphi =i(zx/t+2z^2 )$,  we can get  stationary phase point $ \xi=-\frac{x}{4t}$, which satisfies 
\begin{align*}
	\partial_z \varphi(\xi)=0,\quad \partial^2_z \varphi(\xi)\ne0.
\end{align*}

\begin{figure}[h]
	\centering{\includegraphics[width=0.4\linewidth]{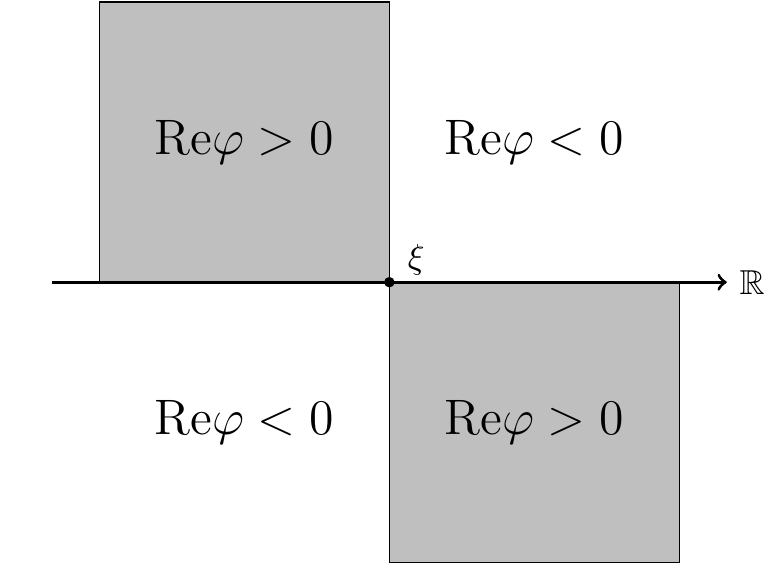}}
	\caption{\label{fig2}The signature table for ${\rm Re}t\varphi$ and $\xi$ the stationary phase point}
\end{figure}

\subsection{The first RH problem transformation}
\indent

Introducing $\delta$:
\begin{align}\label{e18s}
	\delta\equiv \delta(z)=e^{i\int_{-\infty}^{\xi}\frac{ i\nu(s)}{s-z}ds},\quad \nu(s)=-\frac{\log(1-r(s)\breve{r}(s))}{2\pi},
\end{align}
and assuming that $|\arg(1-r(s)\breve r(s))|<\pi$ for $s\in\mathbb{R}$ to secure that $\log(1-r(s)\breve r(s))$ is single-valued,
since $\{r(z),\breve r(z)\}\subset H^1(\mathbb{R})$, we find it possessing properties listed in proposition \ref{p4.1}.
\begin{proposition}\label{p4.1}
	Function $\delta$ admits the following properties:
	\begin{enumerate}
		\item $\delta$ is analytic and non-zero on $\mathbb{C}\setminus(-\infty,\xi]$.
		\item $\delta(z)$ and $\delta(z)^{-1}$ is bounded on $\mathbb{C}\setminus(-\infty,\xi]$.
		\item On $(-\infty, \xi]$, $\delta$ satisfies the jump condition:
		\begin{align*}
			\delta_+=(1-r\breve{r})\delta_-.
		\end{align*}
	\item For any positive number $c<\pi$, as $z\to\infty$ and $|\arg (z-\xi)|\leq c$,
	\begin{align}\label{e18}
		\delta(z)\sim 1-\frac{i}{z}\int_{-\infty}^{\xi}\nu(s)\mathrm{d}s +\mathcal{O}(z^{-2}).
	\end{align}
\item If we write
\begin{align*}
	&\delta(z)=e^{i\beta(z,\xi)}(z-\xi)^{i\nu(\xi)},\\
	&\beta(z,\xi)=\int_{-\infty}^{\xi}\frac{\nu(s)-\chi(x)\nu(\xi)}{z-s}\mathrm{d}s-\nu(\xi)\log(z-\xi+1),
\end{align*}
where $\chi$ is the characterized function of the interval: $[\xi-1,\xi]$, then, at the neighborhood of $z=\xi$, $\beta(z,\xi)$ possesses the asymptotic property:
\begin{align*}
	|\beta(z,\xi)-\beta(\xi,\xi)|\le C(\parallel r\parallel_{H^1}+\parallel r\parallel_{H^1})|z-\xi|^{\frac{1}{2}},\quad z\to\xi.
\end{align*}
	\end{enumerate}
\end{proposition}
\begin{proof}
	Property 1. 2. and 3. is trivial seeing from the definition (\ref{e18s}). For property 4., by taking the Laurent expansion of $(1-s/z)^{1/2}$ at $z\to\infty$, we derive that
	\begin{align*}
		\delta(z)=e^{-\frac{i}{z}\int_{-\infty}^{\xi}\nu(s)\mathrm{d}s+\mathcal{O}(z^{-2})}=I-\frac{i}{z}\int_{-\infty}^{\xi}\nu(s)\mathrm{d}s+\mathcal{O}(z^{-2}).
	\end{align*}
    Finally, we come to property 5., and split $\beta(z,\xi)-\beta(\xi,\xi)$ into three parts:
    \begin{align*}
    	\beta(z,\xi)-\beta(\xi,\xi)=I_1+I_2+I_3,
    \end{align*}
where
    \begin{align*}
    	&I_1=\nu(\xi)\log(z-\xi+1),\quad I_2=\int_{-\infty}^{\xi-1}\frac{\nu(s)}{s-z}\mathrm{d}s-\int_{-\infty}^{\xi-1}\frac{\nu(s)}{s-\xi}\mathrm{d}s,\\
    	&I_3=\int_{\xi-1}^{\xi}\frac{\nu(s)-\nu(\xi)}{s-z}\mathrm{d}s-\int_{\xi-1}^{\xi}\frac{\nu(s)-\nu(\xi)}{s-\xi}\mathrm{d}s.
    \end{align*}
    For $I_1$, we have
    \begin{align*}
    	\log(z-\xi+1)=(z-\xi)+\mathcal{O}((z-\xi)^{2}).
    \end{align*}
    For $I_2$, recalling that both $r(z)$ and $\breve{r}(z)$ are continuous and bounded on the real line, we have
    \begin{align*}
    	 |I_2|&=|2\pi(\mathcal{C}\nu(z)-\mathcal{C}\nu(\xi))|=|2\pi\int_{\xi}^{z}\left(\frac{\mathrm{d}}{\mathrm{d}s}\mathcal{C}\nu\right)(s)\mathrm{d}s|=2\pi|\int_{\xi}^{z}\mathcal{C}\nu'(s)\mathrm{d}s|\\
    	&\le 2\pi\parallel \mathcal{C}\nu'\parallel_2|z-\xi|^\frac{1}{2}\lesssim(\parallel r\parallel_{H^1}+\parallel \breve{r}\parallel_{H^1})|z-\xi|^\frac{1}{2},
    \end{align*}
    where $\mathcal{C}$ is the Cauchy integral operator that is bounded from $L^2$ to $L^2$ on interval $(-\infty,\xi-1)$:
    \begin{align*}
    	\mathcal{C}f(z)=\frac{1}{2\pi i}\int_{-\infty}^{\xi-1}\frac{f(s)}{s-z}\mathrm{d}s.
    \end{align*}
    Refer to Chapter 7 in \cite{Ablowitz2003complex} for more information about Cauchy integral operators.
    For $I_3$, recalling that $r(z)$ is $\frac{1}{2}$-H\"older continuous on the real line, we apply the Cauchy integral operator on the interval $[\xi-1,\xi]$ and have similar estimate for $I_3$:
    \begin{align*}
    	|I_3|\lesssim(\parallel r\parallel_{H^1}+\parallel\breve{r}\parallel_{H^1})|z-\xi|^\frac{1}{2}.
    \end{align*}
Finally, we complete the proof.
\end{proof}

Defining a $2\times2$ matrix function $M^{(1)}=M\delta^{-\sigma_3}$, observing RH problem \ref{rh3.1} and Proposition \ref{p4.1}, we obtain that $M^{(1)}$ solve the following RH problem.
\begin{rhp}\label{rh5.2}
	Find a $2\times2$ matrix function on $\mathbb{C}\setminus\mathbb{R}$, such that:
	\begin{itemize}
		\item Analyticity: $M^{(1)}$ is holomorphic on $\mathbb{C}\setminus\mathbb{R}$.
		\item Normalization:
		\begin{align*}
			M^{(1)}\sim I+\mathcal{O}(z^{-1}),\quad\text{as}\quad z\to\infty.
		\end{align*}
		\item Jump condition:
		\begin{align*}
			M^{(1)}_+=M^{(1)}_-e^{t\varphi\hat\sigma_3}V^{(1)},\quad\text{on}\quad \mathbb{R},
		\end{align*}
		where
		\begin{align*}
			V^{(1)}=\begin{cases}
				\left(\begin{matrix}
					1-r\breve{r}&-\breve{r}\delta^{2}\\r\delta^{-2}&1
				\end{matrix}\right)=
				\left(\begin{matrix}
					1&-\breve{r}\delta^{2}\\0&1
				\end{matrix}\right)
				\left(\begin{matrix}
					1&0\\r\delta^{-2}&1
				\end{matrix}\right)\quad\text{on}\quad (\xi,+\infty),\\
				\left(\begin{matrix}
					1&\frac{-\breve{r}\delta_+^{2}}{1-r\breve{r}}\\\frac{r\delta_-^{-2}}{1-r\breve{r}}&1-r\breve{r}
				\end{matrix}\right)=
				\left(\begin{matrix}
					1&0\\\frac{r\delta_-^{-2}}{1-r\breve{r}}&1
				\end{matrix}\right)
				\left(\begin{matrix}
					1&\frac{-\breve{r}\delta_+^{2}}{1-r\breve{r}}\\0&1
				\end{matrix}\right)\quad\text{on}\quad (-\infty,\xi).
			\end{cases}
		\end{align*}
	\end{itemize}
\end{rhp}

\subsection{$\bar\partial$-RH problem}
\indent

Then, we make another transformation: $M^{(1)}\rightsquigarrow M^{(2)}$, where $M^{(2)}$ admits a $\bar\partial$-RH problem,  and deform the jump contour $\mathbb{R}$ into the contour $\Sigma$ consisting of four rays: $\Sigma_j=\xi+e^{\frac{i\pi}{4}(2j-1)}\mathbb{R}^+$, $j=1,\dots,4$. See more detail of $\Sigma$ at Figure \ref{fig3}
\begin{figure}[h]
	\centering{\includegraphics[width=0.4\linewidth]{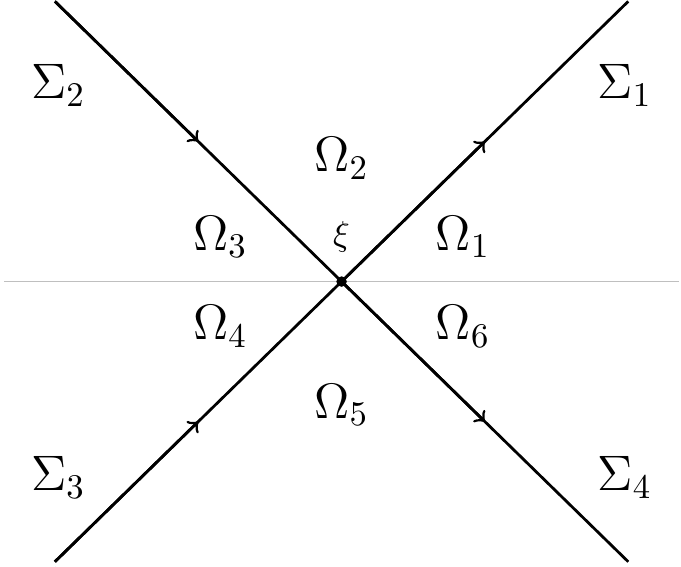}}
	\caption{\label{fig3}The jump contour $\Sigma=\bigcup_{j=1}^4\Sigma_j$ for $\bar\partial$-RH problem \ref{drhp2} and regions $\Omega_j$, $j=1,\dots,6$}
\end{figure}
\begin{lemma}\label{l5.3}
	We can find scalar functions $R_j$: $\Omega_j\to\mathbb{C}$ with the boundary condition:
	\begin{subequations}\label{e17r}
		\begin{align}
			R_1(z)&=\begin{cases}
				r(z)\delta^{-2}(z), \quad z\in(\xi,+\infty),\\
				r(\xi)\delta_0(\xi)^{-2}(z-\xi)^{-2i\nu(\xi)},\quad z\in\Sigma_1,
			\end{cases}\\
		    R_3(z)&=\begin{cases}
		    	\frac{\breve{r}(z)\delta_+^{2}(z)}{1-r(z)\breve{r}(z)}, \quad z\in(-\infty,\xi),\\
		    	\frac{\breve{r}(\xi)\delta^{2}_0(\xi)(z-\xi)^{2i\nu(\xi)}}{1-r(\xi)\breve{r}(\xi)}, \quad z\in\Sigma_2,
		    \end{cases}\\
	        R_4(z)&=\begin{cases}
	        	\frac{r(z)\delta_-^{-2}(z)}{1-r(z)\breve{r}(z)}, \quad z\in(-\infty,\xi),\\
	        	\frac{ r(\xi)\delta_0(\xi)^{-2}(z-\xi)^{-2i\nu(\xi)}}{1-r(\xi)\breve{r}(\xi)}, \quad z\in\Sigma_3,
	        \end{cases}\\
            R_6(z)&=\begin{cases}
            	\breve{r}(z)\delta^{2}(z), \quad z\in(\xi,+\infty),\\
            	\breve{r}(\xi)\delta_0^{2}(\xi)(z-\xi)^{2i\nu(\xi)},\quad z\in\Sigma_4,
            \end{cases}
		\end{align}
	\end{subequations}
such that for $j=1,3,4,6,$
\begin{subequations}\label{e20}
	\begin{align}
		&|R_j(z)|\le C(\sin^2 (\arg(z-\xi))+\big<{\rm Re} z\big>^{-\frac{1}{2}}),\label{e20a}\\
		&|\bar\partial R_j(z)|\le C(|r'({\rm Re} z)|+|\breve{r}'({\rm Re} z)|+|z-\xi|^{-1/2}),
	\end{align}
\end{subequations}
where $\big<\cdot\big>=\sqrt{1+(\cdot)^2}$ and $\delta_0(\xi)=e^{i\beta(\xi,\xi)}$.
\end{lemma}
\begin{proof}
	In (\ref{e20}), without loss of generality, we only check the case of $j=1,3$ and the estimates for other cases similarly follow.  For $j=1$,  rewriting $z=x+iy=se^{i\psi}+\xi$ and defining
	\begin{align*}
		f_1(z)&=r(\xi)\delta^2(z)\delta_0(\xi)^{-2}(z-\xi)^{-2i\nu(\xi)},\\
		R_1(z)&=(r({\rm Re} z)\cos(2\psi)+f_1(z)(1-\cos(2\psi)))\delta^{-2}(z),
	\end{align*}
which obviously satisfies the boundary condition (\ref{e17r}), we can see from Proposition \ref{p4.1} that
\begin{align*}
	f_1(z)=r(\xi)e^{i(\beta(z,\xi)-\beta(\xi,\xi))}
\end{align*}
is bounded on $\Omega_1$; therefore, with the fact that $|r(x)|\lesssim \big<x\big>^{-\frac{1}{2}}$ in (\ref{e11s}), the inequality (\ref{e20a}) is verified.
Because of the fact that $\bar\partial_z=\frac{1}{2}(\partial_x+i\partial_y)=\frac{e^{i\psi}}{2}(\partial_s+is^{-1}\partial_\psi)$, it follows that
	\begin{align}
		|\bar\partial R_1(z)|=&\Big|\frac{1}{2}r'(x)\cos(2\psi)+\frac{e^{i\psi}}{2}(r(x)-f_1(z))\frac{\sin(2\psi)}{|z-\xi|}\Big||\delta^{-2}(z)|\notag\\
		\lesssim& |r'(x)|+\frac{|r({\rm Re} z)-r(\xi)|+|r(\xi)-f_1(z)|}{|z-\xi|}. \label{e21}
	\end{align}
Noticing that by Proposition \ref{p4.1},
\begin{align*}
	|r(x)-R(\xi)|&=\Big|\int_{\xi}^{x}r'(s)\mathrm{d}s\Big|\le \parallel r\parallel_{H^1}|z-\xi|^{\frac{1}{2}},\\
	|r(\xi)-f_1(z)|&=|r(\xi)||1-e^{i(\beta(z,\xi)-\beta(\xi,\xi))}|
	\lesssim r(\xi)||\beta(z,\xi)-\beta(\xi,\xi)|\lesssim |z-\xi|^\frac{1}{2},
\end{align*}
we consequently obtain the estimate for $|\bar\partial R_1(z)|$ from (\ref{e21}):
\begin{align*}
	|\bar\partial R_1(z)|\lesssim |r(x)|+|z-\xi|^{-\frac{1}{2}}.
\end{align*}
For $j=3$, defining
\begin{align*}
	f_3(z)&=\frac{\breve{r}(\xi)}{1-r(\xi)\breve{r}(\xi)}\delta^{-2}(z)\delta_0(\xi)^{2}(z-\xi)^{2i\nu(\xi)},\\
	R_3(z)&=\frac{\breve r({\rm Re} z)}{1-r({\rm Re} z)\breve{r}({\rm Re} z)}\cos(2\psi)+f_3(z)(1-\cos(2\psi)),
\end{align*}
by similar computation, we obtain the estimate (\ref{e20}) for $j=3$.
\end{proof}

Now, we construct a $2\times2$ matrix function $\mathcal{R}$ on $\bigcup_{j=1}^6\Omega_j$:
\begin{align}\label{e37}
	\mathcal{R}(z)=\begin{cases}
		\left(\begin{matrix}
			1&0\\0&1
		\end{matrix}\right),\quad z\in\Omega_2\cup\Omega_5,\\
	    \left(\begin{matrix}
	    	1&0\\-R_1(z)e^{-2t\varphi(z)}&1
	    \end{matrix}\right),\quad z\in\Omega_1,\\
        \left(\begin{matrix}
        	1&R_3(z)e^{2t\varphi(z)}\\0&1
        \end{matrix}\right),\quad z\in\Omega_3,\\
        \left(\begin{matrix}
        	1&0\\R_4(z)e^{-2t\varphi(z)}&1
        \end{matrix}\right),\quad z\in\Omega_4,\\
        \left(\begin{matrix}
        	1&-R_6(z)e^{2t\varphi(z)}\\0&1
        \end{matrix}\right),\quad z\in\Omega_6,
	\end{cases}
\end{align}
and introduce the second RH problem transformation:
\begin{align}\label{e38}
	M^{(2)}=M^{(1)}\mathcal{R},
\end{align}
where we can see from Figure \ref{fig2} that $\mathcal{R}$ decays to the unit matrix as $t\to\infty$.
Seeing from RH problem \ref{rh5.2}, Lemma \ref{l5.3} and (\ref{e38}), someone obtains that $M^{(2)}$ admits the following $\bar\partial$-RH problem.
\begin{drhp}\label{drhp2}
	Find a $2\times2$ matrix function on $\mathbb{C}\setminus\Sigma$ such that:
	\begin{itemize}
		\item Continuity: $M^{(2)}\in C^1(\mathbb{C}\setminus\Sigma)$.
		
		\item Jump condition: On $\Sigma$,
		\begin{align*}
			M^{(2)}_+&=M^{(2)}_-V^{(2)},\quad
			V^{(2)}=\begin{cases}
				\left(\begin{matrix}
					1&0\\R_1(z)e^{-2t\varphi(z)}&1
				\end{matrix}\right),\quad z\in\Sigma_1,\\
			    \left(\begin{matrix}
			    	1&-R_3(z)e^{2t\varphi(z)}\\0&1
			    \end{matrix}\right),\quad z\in\Sigma_2,\\
		        \left(\begin{matrix}
		        	1&0\\R_4(z)e^{-2t\varphi(z)}&1
		        \end{matrix}\right),\quad z\in\Sigma_3,\\
	            \left(\begin{matrix}
	            	1&-R_6(z)e^{2t\varphi(z)}\\0&1
	            \end{matrix}\right),\quad z\in\Sigma_4.
			\end{cases}
		\end{align*}
	\item Normalization: $M^{(2)}\sim I+\mathcal{O}(z^{-1})$ as $z\to\infty$.
	\item $\bar\partial$-condition: For $\mathbb{C}\setminus\Sigma$, we have that
	\begin{align*}
		\bar\partial M^{(2)}=M^{(2)}\mathcal{R}.
	\end{align*}
	\end{itemize}
\end{drhp}

\subsection{The factorization of $\bar\partial$-RH problem}
\indent

For the sake of that asymptotic analysis for the $\bar\partial$-RH problem is fairly complicate, we shall factorize it into the product of $M^{(2)}_{\text{\tiny{RHP}}}$ and $M^{(3)}$
\begin{align}\label{e25}
	M^{(2)}=M^{(3)}M^{(2)}_{\text{\tiny{RHP}}},
\end{align}
where $M^{(2)}_{\text{\tiny{RHP}}}\equiv M^{(2)}_{\text{\tiny{RHP}}}(z;x,t)$ admits RH problem \ref{rhp2} and $M^{(3)}\equiv M^{(3)}(z;x,t)$ is the solution of $\bar\partial$-problem \ref{d3}.
\begin{rhp}\label{rhp2}
	Find a $2\times2$ matrix function $M^{(2)}_{\text{\tiny{RHP}}}$ holomorphic on $\mathbb{C}\setminus\Sigma$ and satisfying the normalization and jump condition of $\bar\partial$-RH problem \ref{drhp2}.
\end{rhp}
\begin{dbarproblem}\label{d3}
	Find a $2\times2$ matrix function on $\mathbb{C}$ such that:
	\begin{itemize}
		\item Continuity: $M^{(3)}\in C^0(\mathbb{C})\cap C^1(\mathbb{C}\setminus\Sigma)$.
		\item Normalization: $M^{(3)}\sim I+\mathcal{O}(z^{-1})$ as $z\to\infty$.
		\item $\bar\partial$-condition:
		\begin{align*}
			\bar\partial M^{(3)}=M^{(3)}W,
		\end{align*}
	where $W=M^{(2)}_{\text{\tiny{RHP}}}\bar\partial\mathcal{R}(M^{(2)}_{\text{\tiny{RHP}}})^{-1}$.
	\end{itemize}
\end{dbarproblem}

\begin{remark}
	To see the well-definedness of factorization (\ref{e25}), we assume the solvability of RH problem \ref{rhp2} and the existence of $M^{(2)}_{\text{\tiny{RHP}}}$ that will be proven at section \ref{S4.4}; then
	\begin{align*}
		M^{(3)}=M^{(2)}({M^{(2)}_{\text{\tiny{RHP}}}})^{-1}
	\end{align*}
    is well-defined. Seeing from $\bar\partial$-RH problem \ref{drhp2} and RH problem \ref{rhp2}, $M^{(3)}$ does satisfy $\bar\partial$ problem \ref{d3}.
\end{remark}

\begin{figure}[h]
	\centering{\includegraphics[width=0.4\linewidth]{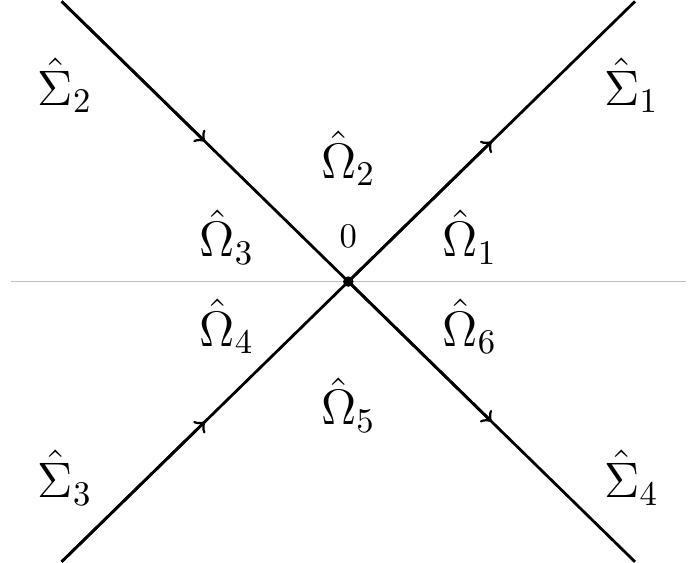}}
	\caption{\label{fig4}The jump contour $\hat\Sigma=\bigcup_{j=1}^4\hat\Sigma_j$ for RH problem \ref{rhp4.8s} and regions $\hat\Omega_j$, $j=1,\dots,6$}
\end{figure}
\subsection{the model RH problem}\label{S4.4}
\indent
	
To solve RH problem \ref{rhp2}, we first make a scaling transformation for $M^{(2)}_\text{\tiny{RHP}}$:
\begin{align}
	\hat M\equiv \hat M(\zeta)=M^{(2)}_\text{\tiny{RHP}}(z),\quad z=\zeta/\sqrt{8t}+\xi\label{e40},
\end{align}
then, by RH problem \ref{rhp2} and (\ref{e40}), $\hat M$ solves the RH problem:
\begin{rhp}\label{rhp4.8s}
	Find a $2\times2$ matrix function on $\mathbb{C}\setminus\hat\Sigma$, where $\hat\Sigma=\Sigma-\xi$ and see more detail about $\hat\Sigma$ at Figure \ref{fig4}, such that:
	\begin{itemize}
		\item Analyticity: $\hat M$ is holomorphic on $\mathbb{C}\setminus \hat \Sigma$.
		\item Normalization: $\hat M(\zeta)\sim I+\mathcal{O}(1/\zeta)$, as $\zeta\to\infty$.
		\item Jump condition: On $\hat\Sigma$, we have
		\begin{align*}
			\hat M_+=\hat M_-\hat V,
		\end{align*}
	where $\hat V\equiv\hat V(\zeta)=V^{(2)}(\zeta/\sqrt{8t}+\xi)$ and by basic computation, $\hat V$ can be written explicitly:
	\begin{align*}
		\hat V(\zeta)&=e^{\frac{i\zeta^2}{4}\hat\sigma_3}\zeta^{i\nu(\xi)\hat\sigma_3}\hat V_0,\\
	\hat V_0&=\begin{cases}
		\left(\begin{matrix}
			1&0\\r(\xi)(\delta_0(\xi))^{-2}(8t)^{i\nu(\xi)}e^{\frac{ix^2}{4t}}&1
		\end{matrix}\right),\quad \zeta\in\hat\Sigma_1,\\
	    \left(\begin{matrix}
	    	1&-\frac{\breve{r}(\xi)}{1-r(\xi)\breve{r}(\xi)}(\delta_0(\xi))^{2}(8t)^{-i\nu(\xi)}e^{-\frac{ix^2}{4t}}\\0&1
	    \end{matrix}\right),\quad \zeta\in\hat\Sigma_2,\\
        \left(\begin{matrix}
        	1&0\\\frac{r(\xi)}{1-r(\xi)\breve{r}(\xi)}(\delta_0(\xi))^{-2}(8t)^{i\nu(\xi)}e^{\frac{ix^2}{4t}}&1
        \end{matrix}\right),\quad \zeta\in\hat\Sigma_3,\\
        \left(\begin{matrix}
        	1&-\breve{r}(\xi)(\delta_0(\xi))^{2}(8t)^{-i\nu(\xi)}e^{-\frac{ix^2}{4t}}\\0&1
        \end{matrix}\right),\quad \zeta\in\hat\Sigma_4.\\
	\end{cases}
	\end{align*}
	\end{itemize}
\end{rhp}

Set
\begin{align}\label{e31}
	\Psi\equiv\Psi(\zeta)&=\hat M(\zeta)e^{\frac{i\zeta^2}{4}\sigma_3}\zeta^{i\nu(\xi)\sigma_3}P_0,
\end{align}
where $P_0$ is a $2\times2$ matrix function that is constant on each $\hat\Omega_j$ for $j=1,\dots,6$:
\begin{align*}
	P_0&=\begin{cases}
		\left(\begin{matrix}
			1&0\\0&1
		\end{matrix}\right), \quad\zeta\in\hat\Omega_2\cup\hat\Omega_5,\\
		\left(\begin{matrix}
			1&0\\r(\xi)\delta_0(\xi)^{-2}(8t)^{i\nu(\xi)}e^{\frac{ix^2}{4t}}&1
		\end{matrix}\right), \quad\zeta\in\hat\Omega_1,\\
		\left(\begin{matrix}
			1&-\frac{\breve{r}(\xi)}{1-r(\xi)\breve{r}(\xi)}\delta_0(\xi)^{2}(8t)^{-i\nu(\xi)}e^{-\frac{ix^2}{4t}}\\0&1
		\end{matrix}\right), \quad\zeta\in\hat\Omega_3,\\
		\left(\begin{matrix}
			1&0\\-\frac{r(\xi)}{1-r(\xi)\breve{r}(\xi)}\delta_0(\xi)^{-2}(8t)^{i\nu(\xi)}e^{\frac{ix^2}{4t}}&1
		\end{matrix}\right), \quad\zeta\in\hat\Omega_4,\\
		\left(\begin{matrix}
			1&\breve{r}(\xi)\delta_0(\xi)^{2}(8t)^{-i\nu(\xi)}e^{-\frac{ix^2}{4t}}\\0&1
		\end{matrix}\right), \quad\zeta\in\hat\Omega_6,
	\end{cases}
\end{align*}
then, it's trivial to verify that $\Psi$ admits RH problem \ref{rhp4.8}.
\begin{rhp}\label{rhp4.8}
	Find a $2\times2$ matrix function on $\mathbb{C}\setminus\mathbb{R}$ such that
	\begin{itemize}
		\item Analyticity: $\Psi$ is holomorphic on $\mathbb{C}\setminus\mathbb{R}$.
		\item normalization: $\Psi(\zeta)e^{-\frac{i\zeta^2}{4}\sigma_3}\zeta^{-i\nu(\xi)\sigma_3}\sim I$ as $\zeta\to\infty$.
		\item Jump condition: The boundary value of $\Psi$ on $\mathbb{R}$ satisfies:
		\begin{align*}
			&\Psi_+=\Psi_-V_\xi, \\
			&V_\xi=\left(
			\begin{matrix}
				1-r(\xi)\breve{r}(\xi)&-\breve{r}(\xi)\delta_0(\xi)^{2}(8t)^{-i\nu(\xi)}e^{-\frac{ix^2}{4t}}\\r(\xi)\delta_0(\xi)^{-2}(8t)^{i\nu(\xi)}e^{\frac{ix^2}{4t}}&1
			\end{matrix}\right).
		\end{align*}
	\end{itemize}
\end{rhp}
Finally, like Step 8 in \cite{deift1993long}, we get a kind of RH problem that we can change it into Weber equation and we obtain the solution in terms of parabolic cylinder functions.
\begin{proposition}\label{p4.10}
	RH problem \ref{rhp4.8} is solvable and the solution $\Psi$ is given by
	\begin{align*}
		\Psi=\left(\begin{matrix}
			\Psi_{1,1}&\Psi_{1,2}\\\Psi_{2,1}&\Psi_{2,2}
		\end{matrix}\right),
	\end{align*}
    where
    \begin{align*}
    	\Psi_{1,1}(\zeta)&=\begin{cases}
    		e^{\frac{\pi\nu}{4}}D_{-i\nu}(e^{-\frac{\pi i}{4}}\zeta),\quad {\rm Im}\zeta>0,\\
    		e^{-\frac{3\pi\nu}{4}}D_{-i\nu}(e^{\frac{3\pi i}{4}}\zeta),\quad {\rm Im}\zeta<0,
    	\end{cases}\\
    	\Psi_{1,2}(\zeta)&=\begin{cases}
    		e^{-\frac{3\pi\nu}{4}}(\beta_1)^{-1}[\partial_\zeta D_{i\nu}(e^{-\frac{3\pi i}{4}}\zeta)+\frac{i\zeta}{2}D_{i\nu}(e^{-\frac{3\pi i}{4}}\zeta)],\quad{\rm Im}\zeta>0,\\
    		e^{\frac{\pi\nu}{4}}(\beta_1)^{-1}[\partial_\zeta D_{i\nu}(e^{\frac{\pi i}{4}}\zeta)+\frac{i\zeta}{2}D_{i\nu}(e^{\frac{\pi i}{4}}\zeta)],\quad{\rm Im}\zeta<0,
    	\end{cases}\\
        \Psi_{2,1}(\zeta)&=\begin{cases}
        	e^{\frac{\pi\nu}{4}}(\beta_2)^{-1}[\partial_\zeta D_{-i\nu}(e^{-\frac{\pi i}{4}}\zeta)-\frac{i\zeta}{2}D_{-i\nu}(e^{-\frac{\pi i}{4}}\zeta)],\quad{\rm Im}\zeta>0,\\
        	e^{-\frac{3\pi\nu}{4}}(\beta_2)^{-1}[\partial_\zeta D_{-i\nu}(e^{\frac{3\pi i}{4}}\zeta)-\frac{i\zeta}{2}D_{-i\nu}(e^{\frac{3\pi i}{4}}\zeta)],\quad{\rm Im}\zeta<0,
        \end{cases}\\
        \Psi_{2,2}(\zeta)&=\begin{cases}
        	e^{-\frac{3\pi\nu}{4}}D_{i\nu}(e^{-\frac{3\pi i}{4}}\zeta),\quad {\rm Im}\zeta>0,\\
        	e^{\frac{\pi\nu}{4}}D_{i\nu}(e^{\frac{\pi i}{4}}\zeta),\quad {\rm Im}\zeta<0,
        \end{cases}
    \end{align*}
    \begin{align*}
    	\beta_1&=\frac{\sqrt{2\pi}e^{\frac{\pi i}{4}}e^{-\frac{\pi\nu}{2}}}{\rho_0\Gamma(-i\nu)}, &\beta_2&=\frac{\nu}{\beta_1},&\rho_0=-\breve r(\xi)\delta_0(\xi)^{2}(8t)^{-i\nu}e^{-\frac{ix^2}{4t}},\quad \nu\equiv\nu(\xi),
    \end{align*}
and $D_a(\eta)$ is a solution of the Weber equation
\begin{align*}
	\partial_{\eta}^2D_a(\eta)+\left[\frac{1}{2}-\frac{\eta^2}{4}+a\right]D_a(\eta)=0.
\end{align*}
\end{proposition}
\begin{remark}\label{r5.11}
	In addition to Proposition \ref{p4.10}, if we write $\hat M$ as
	\begin{align}\label{e39}
		\hat M(\zeta)=I+\zeta^{-1}\hat M_{-1}+\mathcal{O}(\zeta^{-2}),
	\end{align}
    then
    \begin{align*}
    	\beta_1=i(\hat M_{-1})_{1,2},\quad \beta_2=-i(\hat M_{-1})_{2,1}.
    \end{align*}
\end{remark}

\subsection{Analysis on a $\bar\partial$-problem}
\indent

In this section, according to $\bar\partial$-problem \ref{d3}, we obtain an integral equation (\ref{e42r}) for it; then, we make some estimates based on the integral operator $\jmath$ defined by (\ref{e44}).

$\bar\partial$-problem \ref{d3} is equivalent to the integral equation:
\begin{align}\label{e42r}
	M^{(3)}(z)=I+\frac{1}{\pi}\iint_\mathbb{C}\frac{M^{(3)}(s) W(s)}{z-s}\mathrm{d^2}s,\quad z\in\mathbb{C},
\end{align}
which is equivalent to
\begin{align}\label{e45r}
	(I-\jmath)M^{(3)}=I,
\end{align}
where $\jmath$ is an integral operator such that for a $2\times2$ matrix function $f$,
\begin{align}\label{e44}
	\jmath [f](z)=\frac{1}{\pi}\iint_\mathbb{C}\frac{f(s)W(s)}{z-s}\mathrm{d^2}s,\quad z\in\mathbb{C}.
\end{align}
To derive the solvability of $\bar\partial$-problem \ref{d3}, we prove that when $t$ is sufficiently large, $\parallel\jmath\parallel_{L^\infty\to L^\infty}$ is small and the resolvent operator $(1-\jmath)^{-1}$ exists on $L^\infty(\mathbb{C})$.
\begin{proposition}\label{p5.12}
	For $t>0$, there is a constant $C>0$, such that
	\begin{align*}
		\parallel\jmath\parallel_{L^\infty\to L^\infty}\le Ct^{-1/4}.
	\end{align*}
\end{proposition}
\begin{proof}
	We only consider the case of matrix function with support region in $\hat\Omega_1$, and cases for other regions follow in the similar way. Setting $f(s)\in L^\infty(\mathbb{C})$ and $s=u+iv$, we obtain the modular estimate for $z=x+iy\in\mathbb{C}$,
	\begin{align}
	|\jmath[f](z)|\le&\frac{1}{\pi}\iint_{\hat\Omega_1}\frac{|f(s)M^{(2)}_\text{\tiny{RHP}}(s)\bar\partial\mathcal{R}(s)M^{(2)}_\text{\tiny{RHP}}(s)^{-1}|}{|z-s|}\mathrm{d}s\notag\\
		\le& \frac{\parallel f\parallel_{L^\infty}}{\pi}\parallel M^{(2)}_\text{\tiny{RHP}}\parallel_{L^\infty}\parallel {M^{(2)}_\text{\tiny{RHP}}}^{-1}\parallel_{L^\infty}\iint_{\hat\Omega_1}\frac{|\bar\partial R_1(s)|e^{-8tv(u-\xi)}}{|s-z|}\mathrm{d^2}s.\label{e41r}
	\end{align}
Seeing from RH problem \ref{rhp2}, we find that $\det M^{(2)}_{\text{\tiny{RHP}}}\equiv1$, i.e., $M^{(2)}_{\text{\tiny{RHP}}}$ is invertible;
then, since $M^{(2)}_{\text{\tiny{RHP}}}$ has no pole on the complex plane, $\parallel M^{(2)}_{\text{\tiny{RHP}}}\parallel_{L^\infty}$ and $\parallel {M^{(2)}_{\text{\tiny{RHP}}}}^{-1}\parallel_{L^\infty}$ are bounded;
therefore, by (\ref{e20}) and (\ref{e41r}), we obtain that
\begin{align}
	|\jmath[f](z)|\lesssim(I_4+I_5)\parallel f\parallel_{L^\infty}, \label{e41}
\end{align}
where
\begin{align*}
	I_4&=\iint_{\hat\Omega_1}\frac{(|r'(u)|+ |\breve{r}'(u)|)e^{-8tv(u-\xi)}}{|s-z|}\mathrm{d^2}s,\\
	I_5&=\iint_{\hat\Omega_1}\frac{|s-\xi|^{-\frac{1}{2}}e^{-8tv(u-\xi)}}{|s-z|}\mathrm{d^2}s.
\end{align*}
Since $r(u),\ \breve{r}(u)\in H^1(\mathbb{R})$, we have
\begin{align}
	I_4\le& \int_{0}^{\infty}e^{-tv^2}\int_{v+\xi}^\infty\frac{|r'(u)|+|\breve r'(u)|}{|s-z|}\mathrm{d}u\mathrm{d}v\notag\\
	\le&(\parallel r'\parallel_{L^2}+\parallel \breve{r}'\parallel_{L^2})\int_{0}^{\infty}e^{-tv^2}\mathrm{d}v\left(\int_{v+\xi}^{\infty}|s-z|^{-2}\mathrm{d}u\right)^\frac{1}{2}\notag\\
	\lesssim&\int_{0}^{\infty}\frac{e^{-tv^2}}{|v-y|^\frac{1}{2}}\mathrm{d}v\lesssim\int_{0}^{\infty}\frac{e^{-tv^2}}{v^{\frac{1}{2}}}\lesssim t^{-\frac{1}{4}}. \label{e42s}
\end{align}
For the boundedness of $I_2$, by H\"older's inequality, we have
\begin{align}
	I_5&\le\int_{0}^{\infty}e^{-8tv^2}\int_{v+\xi}^{\infty}\frac{\mathrm{d}u\mathrm{d}v}{|s-\xi|^\frac{1}{2}|s-z|}\notag\\
	&\le \int_{0}^{\infty}e^{-8tv^2}\left(\int_{v+\xi}^{\infty}|s-\xi|^{-\frac{p}{2}}\right)^\frac{1}{p}\left(\int_{v+\xi}^{\infty}|s-z|^{-q}\right)^\frac{1}{q}\notag\\
	&\lesssim \int_{0}^{\infty}e^{-tv^2}v^{\frac{1}{p}-\frac{1}{2}}|v-y|^{\frac{1}{q}-1}\mathrm{d}v. \label{e42}
\end{align}
On the one hand, the right side of (\ref{e42}) is obviously less than  $\int_{0}^{\infty}\frac{e^{-tv^2}}{v^{\frac{1}{2}}}\mathrm{d}v\lesssim t^{-\frac{1}{4}}$ as $y\le0$; on the other hand, when $y>0$,
\begin{align*}
	&\int_{0}^{\infty}e^{-tv^2}v^{\frac{1}{p}-\frac{1}{2}}|v-y|^{\frac{1}{q}-1}\mathrm{d}v\\
	=&\int_{0}^{y}e^{-tv^2}v^{\frac{1}{p}-\frac{1}{2}}(y-v)^{\frac{1}{q}-1}\mathrm{d}v
	+\int_{y}^{\infty}e^{-tv^2}v^{\frac{1}{p}-\frac{1}{2}}(v-y)^{\frac{1}{q}-1}\mathrm{d}v\\
	\le&t^{-\frac{1}{4}}\int_{0}^{1}e^{-8tv^2y^2}(tv^2y^2)^{\frac{1}{4}}v^{\frac{1}{p}-1}(1-v)^{\frac{1}{q}-1}\mathrm{d}v+\int_0^\infty e^{-8tv^2}v^{-\frac{1}{2}}\mathrm{d}v\lesssim t^{-\frac{1}{4}},
\end{align*}
which yield
\begin{align}\label{e43}
	I_5\lesssim t^{-\frac{1}{4}}.
\end{align}
With (\ref{e41}, \ref{e42s}, \ref{e43}), the result is verified.
\end{proof}
It's a consequence of (\ref{e45r}) and Proposition \ref{p5.12} that for $t>0$ large enough, the $L^{\infty}$-norm of $M^{(3)}$ is bounded.
Since we have confirmed the solvability of $\bar\partial$-problem for large $t$, we now determine the long-time asymptotic behavior of the second coefficient in the Laurent expansion for $M^{(3)}$:
\begin{align}
	&M^{(3)}(z)=I+M^{(3)}_{-1}z^{-1}+\mathcal{O}(z^{-2}),\notag\\
	&M^{(3)}_{-1}=\frac{1}{\pi}\iint_\mathbb{C}M^{(3)}(s)W(s)\mathrm{d^2}s. \label{e45}
\end{align}
\begin{proposition}\label{p5.13}
	For $t>0$, the coefficient $M^{(3)}_{-1}$ satisfies that
	\begin{align*}
		\vert M^{(3)}_{-1}\vert\lesssim t^{-\frac{3}{4}}.
	\end{align*}
\end{proposition}
\begin{proof}
	To bound the norm, without loss of generality, we consider the integral region in (\ref{e45}) as $\Omega_1$; then, by the boundedness of $\parallel M^{(3)}\parallel_{L^\infty}$, $\parallel M^{(2)}_{\text{\tiny{RHP}}}\parallel_{L^\infty}$ and $\parallel {M^{(2)}_{\text{\tiny{RHP}}}}^{-1}\parallel_{L^\infty}$, we obtain
	\begin{align}
		|M^{(3)}_{-1}|&\le \frac{1}{\pi}\parallel M^{(3)}\parallel_{L^\infty}\parallel M^{(2)}_{\text{\tiny{RHP}}}\parallel_{L^\infty}\parallel {M^{(2)}_{\text{\tiny{RHP}}}}^{-1}\parallel_{L^\infty}\iint_{\Omega_1}|\bar\partial \mathcal{R}(s)|\mathrm{d^2}s \notag\\
		&\lesssim \iint_{\Omega_1}|\bar\partial R_1(s)e^{-2t\varphi(s)}|\mathrm{d^2}s\le I_6+I_7,\label{e47}
	\end{align}
where
\begin{align*}
	I_6=\iint_{\Omega_1}(|r'(u)|+|\breve{r}'(u)|)e^{-8tv(u-\xi)}\mathrm{d}u\mathrm{d}v,\quad I_7=\iint_{\Omega_1}|s-\xi|^{-\frac{1}{2}}e^{-8tv(u-\xi)}\mathrm{d}u\mathrm{d}v.
\end{align*}
For the boundedness of $I_3$ and $I_4$, by Schwartz inequality and variable substitutions,
\begin{align*}
	I_6&=\int_{0}^{\infty}e^{-8tv^2}\int_{\xi+v}^{\infty}(|r'(u)|+|\breve{r}'(u)|)e^{-8tv(u-v-\xi)}\mathrm{d}u\mathrm{d}v\\
	&\le\int_{0}^{\infty}e^{-8tv^2}(\parallel r'\parallel_{L^2}+\parallel \breve{r}'\parallel_{L^2})\left(\int_{0}^{\infty}e^{-16tuv}\mathrm{d}u\right)^\frac{1}{2}\mathrm{d}v\\
	&\lesssim \int_{0}^{\infty}\frac{e^{-8tv^2}}{4\sqrt{tv}}\mathrm{d}v\lesssim t^{-\frac{3}{4}},\\
	I_7&=\int_{0}^{\infty}e^{-8tv^2}\int_{\xi+v}^{\infty}((u-\xi)^2+v^2)^{-\frac{1}{4}}e^{-8tv(u-v-\xi)}\mathrm{d}u\mathrm{d}v\\
	&\le\int_{0}^{\infty}e^{-8tv^2}\left(\int_{v}^{\infty}(u^2+v^2)^{-1}\mathrm{d}u\right)^\frac{1}{4}\left(\int_{0}^{\infty}e^{-\frac{32}{3}tuv}\mathrm{d}u\right)^\frac{3}{4}\mathrm{d}v\\
	&\lesssim\int_{0}^{\infty}e^{-8tv^2}t^{-\frac{3}{4}}v^{-1}\mathrm{d}v\lesssim t^{-\frac{3}{4}},
\end{align*}
which combined with (\ref{e47}) yield the result.
\end{proof}

\section{Long-time asymptotics of the NNLS equation}\label{S5}
\indent

By the matrix transformation in section \ref{s4}, we learn that
\begin{align*}
	M(z)=M^{(3)}(z)\hat M(\sqrt{8t}(z-\xi))\mathcal{R}^{-1}\delta(z)^{\sigma_3},
\end{align*}
then, taking $z\to\infty$ for $z\in\Omega_2$ and using (\ref{e18}), (\ref{e37}) and (\ref{e39}), we get the asymptotic property of $M$
\begin{align*}
	M\sim I+\left(M^{(3)}_{-1}+\frac{\hat M_{-1}}{\sqrt{8t}}-\left(i\int_{-\infty}^{\xi}\nu(s)\mathrm{d}s\right)^{\sigma_3}\right)z^{-1}+\mathcal{O}(z^{-2}),
\end{align*}
which combined with (\ref{e17}), Remark \ref{r5.11} and Proposition \ref{p5.13} yields that when $\arg(1-r(\xi)\breve{r}(\xi))>-\frac{\pi}{2}$, i.e., ${\rm Im}\nu(\xi)<\frac{1}{4}$,
\begin{align*}
	q(x,t)&=2\left(\frac{\beta_1}{\sqrt{8t}}+i(M^{(3)}_{-1})_{1,2}\right)=\alpha(\xi)t^{-{\rm Im}\nu-\frac{1}{2}}+\mathcal{O}(t^{-\frac{3}{4}}),\\
	\alpha(\xi)&=-\frac{\sqrt{\pi}e^{-\frac{\pi \nu}{2}+\frac{\pi i}{4}+4it\xi^2}t^{i{\rm Re}\nu}}{\breve r(\xi)8^{-i\nu}\delta_0(\xi)^{2}\Gamma(-i\nu)}.
\end{align*}

\end{document}